\documentclass [12pt]{amsart}

\usepackage{graphicx}
\usepackage{amsmath}
\usepackage{amssymb}
\usepackage{enumerate}

\newtheorem{thm}{Theorem}
\newtheorem{Theorem}[thm]{Theorem}
\newtheorem{lemma}{Lemma}

\newtheorem{rmk1}{Remark}

\numberwithin{equation}{section}

\newcommand{\con}{\mathfrak{C}}

\newcommand{\bfrac}[2]{\left(\frac{#1}{#2}\right)}
\newcommand{\tRe}{\textup{Re }}

\newcommand{\G}{\mathcal G}
\newcommand{\T}{\mathcal T}
\newcommand{\sgn}{\textup{sgn}}
\title{Additive twists of fourier coefficients of $GL(3)$ Maass forms}
\author{Xiannan Li}
\date{} 

\begin{document}
\maketitle
\begin{abstract}
We prove cancellation in a sum of Fourier coefficents of a $GL(3)$ form $F$ twisted by additive characters, uniformly in the form $F$.  Previously, this type of result was available only when $F$ is a symmetric square lift.
\end{abstract}

\section{Introduction}
Substantial work has been done in studying sums involving coefficients attached to various $L$-functions.  A very classical example is the problem of estimating exponential sums of the form $\sum_{n\leq x} n^{it}$, which is related to subconvex bounds for the Riemann zeta function\footnote{We may view $n^{it}$ as the coefficients of the Dirichlet series for $\zeta(s-it)$ where $\zeta(s)$ denotes the Riemann zeta function.} and to Dirichlet's divisor problem.  For more on this, see, for instance, Chapters V and XII in \cite{Ti}.  A vast literature also exists for the estimation of character sums.  These are, among other things, related to subconvexity for Dirichlet $L$-functions and estimates for the least quadratic non-residue.  See, for instance, Chapter 12 of \cite{IK}.

The estimation of sums of coefficients twisted by additive characters is also classical.  To be specific, we shall be interested in sums of the type 
$$S = \sum_{n\leq N} a_n e(n \alpha)$$ where as usual, $e(x) = e^{2\pi i x}$.  Here $a_n$ may be the coefficients of certain $L$-functions, or more general coefficients of arithmetic interest.  This type of sum had already appeared in the work of Hardy and Littlewood \cite{HL} in 1914 and has been investigated extensively.  See also the work of Montgomery and Vaughan \cite{MV}.  

In the case of automorphic forms on $GL_2(\mathbb{R})$, obtaining cancellation in $S$ is well understood when the $a_n$ are either the normalized  Fourier coefficients of a modular form, or a Maass form on the upper half plane.  For instance, if $f(z) = \sum_{n} a(n)n^{\frac{k-1}{2}} e(nz)$ is a weight $k$ modular form, then it is not hard to prove that 
$$S\ll_f N^{1/2} \log N,$$ and this is essentially the truth, as can be seen from the $L^2$ norm of $S = S(\alpha)$ for $\alpha \in [0, 1]$.  (See Chapter 5 of \cite{Iw}.)  Note that while the bound depends implicitly on $f$, it is uniform in $\alpha$, which is useful for applications towards proving the same bound for the sum of such coefficients restricted to any arithmetic progression.  Moreover, the proof for this case is fairly straightforward, depending only on an estimate for the size of $f(z)$.

Results on such sums in higher rank settings are quite recent and exhibit new features.  Here, S. D. Miller \cite{Mi} proved the first result and showed that
$$\sum_{n\leq N} A(1, n) e(\alpha n) \ll_F N^{3/4+\epsilon},
$$where $A(m, n)$ are the Fourier coefficients of a cusp form $F$ on $GL(3, \mathbb{Z})\backslash GL(3, \mathbb{R})$, where the result is uniform in $\alpha$, but the implied constant depends on the form $F$.  In the same paper, he discusses the connection between such a bound and bounds on the second moment $\int_{-T}^T |L(1/2+it)|^2dt$ where where $L(s) = L(s, F)$ is the $L$-function attached to $F$.  The main tool used in this proof is Voronoi summation for $GL(3)$ developed by Miller and Schmid \cite{MS}.  

It is natural and sometimes desirable for applications to prove such a bound uniformly in $F$.  In this direction, Xiaoqing Li and M. Young \cite{LY} prove a result in the special case where $F$ is a symmetric square lift of a $SL(2, \mathbb{Z})$ Hecke-Maass form.  Their main result is
$$\sum_{n\leq N} A(1, n) e(\alpha n) \ll N^{3/4+\epsilon} \lambda_F(\Delta)^{D+\epsilon},
$$where $\lambda_F(\Delta)$ is the analytic conductor of $L(s, F)$ and $D = 1/4$ assuming Ramanujan and $D = 1/3$ unconditionally.  The proof is more intricate, depending on a careful technical analysis of exponential integrals which appear in Voronoi.  An interesting new phenomenon which occurs in their work is the localization of the dual sum in very short intervals.  It is for this reason that the Ramanujan conjecture becomes relevant.  

The authors of \cite{LY} restrict their attention to the symmetric square case as a compromise between generality and difficulty.  Symmetric square lifts are a thin subset of all $GL(3, \mathbb{Z})\backslash GL(3, \mathbb{R})$ cusp forms, so it would be interesting to extend this result to general Maass forms.  That is the focus of the present paper.

\begin{Theorem}\label{thm:main}
Let $F$ be a tempered cusp form on $GL(3, \mathbb{Z})\backslash GL(3, \mathbb{R})$ with Fourier coefficients $A(m, n)$, and Langlands parameters $\alpha_i$, $1\leq i\leq 3$.  Let $\con =\prod_{i=1}^3 (1+|\alpha_i|)$.  Then for any $\alpha \in \mathbb{R}$,
$$\sum_{n\leq N} A(1, n) e(n\alpha)\ll_\epsilon N^{3/4+\epsilon} \con^{D},
$$where we may take $D = 1/4$ assuming Ramanujan, and $D = 5/12$ unconditionally.
\end{Theorem}

\begin{rmk1}
$\\$
\begin{enumerate}
\item The quality of the unconditional bound in our result is inferior to the unconditional bound in \cite{LY} due to the presence of functoriality results for $GL(2)$ which can be used for symmetric square lifts.  
\item Here, $\con$ is the usual analytic conductor for $L(1/2, F)$.  It is the same size as $\max (|\lambda_1|, |\lambda_2|)$, where the $\lambda_i$s are the eigenvalues of the Laplace-Casimir operators as defined in \S 6 of \cite{Go}.
\item As mentioned before, the work of Xiaoqing Li and Young \cite{LY} includes an analysis of very short sums in a range like $A\leq n\leq A+B$, where $\frac{B}{A}$ is small.  One of the differences in the general case is that sometimes this short sum behaviour disappears because $A$ can also be very small.  However, this is balanced out by the matching properties of functions appearing in the integral transform.
\end{enumerate}
\end{rmk1}

Rather than bound the sum $\sum_{n\leq N} A(1, n) e(n\alpha)$ directly, it will be more convenient to bound a smooth version of that sum.  

\begin{Theorem}\label{thm:smooth}
Perserve notation as in Theorem \ref{thm:main}.  Let $w$ be a smooth function with support in $[N, 2N]$ and such that $w^{(j)}(y) \ll_j N^{-j}$ for all $j\geq 0$.  Then
$$\sum_{n\geq 1} A(1, n) e(n\alpha) w(n) \ll_\epsilon N^{3/4+\epsilon} \con^{D},
$$where we may take $D = 1/4$ assuming Ramanujan, and $D = 5/12$ unconditionally.
\end{Theorem}
Theorem \ref{thm:main} follows from Theorem \ref{thm:smooth} by standard methods (see \S 9 of \cite{LY}).  We now concentrate on proving Theorem \ref{thm:smooth}.

\section{The basic setup}
First write $\alpha = \frac{a}{q} + \frac{\theta}{2 \pi}$ where $(a, q) = 1$, $q\leq Q$ and $\theta \leq \frac{2\pi }{qQ}$, possible by Dirichlet's theorem on Diophantine approximation. \footnote{$Q$ is a parameter to be determined later.}  We then apply Voronoi summation to 
$$S = \sum_{n\leq N} A(1, n) e(\frac{an}{q}) \psi(n),
$$where $\psi(y) = e^{i\theta y} w(y)$.

The Voronoi summation formula for $GL(3)$ was first proven by Miller and Schmid \cite{MS}, and reproved by Goldfeld and Xiaoqing Li \cite{GL} using an alternate method.  We first introduce some notation.  Let
$$\tilde{\psi}(s) = \int_0^\infty \psi(x) x^s \frac{dx}{x},
$$and 
\begin{equation}\label{eqn:Psi}
\Psi_k(x) = \frac{1}{2\pi i} \int_{(\sigma)} (\pi^3 x)^{-s} \frac{\Gamma \bfrac{1+s+\alpha_1+k}{2}\Gamma \bfrac{1+s+\alpha_2+k}{2}\Gamma \bfrac{1+s+\alpha_3+k}{2}}{\Gamma \bfrac{-s-\alpha_1+k}{2}\Gamma \bfrac{-s-\alpha_2+k}{2}\Gamma \bfrac{-s-\alpha_3+k}{2}} \tilde{\psi}(-s) ds.
\end{equation}

Write $\bar{a}$ for the multiplicative inverse of $a$ modulo $q$.  Further define 
$$\Psi_{\pm}(x) = \frac{1}{2\pi^{3/2}} \left(\Psi_0(x)\pm \frac{1}{i} \Psi_1(x)\right).
$$
Then, by Voronoi summation \cite{MS}, the sum $S = \mathcal S_+ + \mathcal S_-$, where
$$\mathcal S_{\pm} = q \sum_{n_1|q} \sum_{n_2\geq 1} \frac{A(n_2, n_1)}{n_1n_2} S(\bar{a}, \pm n_2; q/n_1) \Psi_{\pm}\bfrac{n_2n_1^2}{q^3}.
$$

It is important to understand the dependence of the integral transforms $\Psi_k$ on the Langlands parameters $\alpha_i$ since this is where the dependence on the conductor arises.  This forms the bulk of the proof.  Before proceeding, we record a few basic results from \cite{LY}.  First, by Lemma 4.1 of \cite{LY},
$$S \ll q^{3/2+\epsilon} \max_{\pm} \max_{d|q} \max_{n_1|q/d} \sum_{n\geq 1} \frac{|A(n, 1)|}{n} \left|\Psi_{\pm}\bfrac{nn_1^2}{(q/d)^3}\right|.
$$The presence of the parameters $d$ and $n_1$ are unimportant to the actual analysis.  Without loss of generality, we will assume that $d=n_1=1$, which will simplify the cluttered notation; the other values of $d$ and $n_1$ can be bounded the same way.  This reduces the problem of bounding $S$ to bounding 
\begin{equation}\label{eqn:T}
\T_k = q^{3/2+\epsilon} \sum_{n\geq 1} \frac{|A(n, 1)|}{n} \left|\Psi_{k}\bfrac{n}{q^3}\right|.
\end{equation}

\subsection{A saddlepoint approximation}
Write $s = \sigma + i\tau$ so that $\tilde{\psi}(s) = x^{\sigma} I$, where
$$I = \int_0^\infty \omega(x)e^{i\theta x}x^{i\tau} \frac{dx}{x}.
$$If the integral is oscillatory, then the saddlepoint method may be applied to evaluate $I$. We quote Lemma 5.1 from \cite{LY} for this purpose.

\begin{lemma}\label{lem:integral}
With notation as above, if $|\tau|\geq 1$ and $|\theta N|  \geq 1$ then 
$$I = \sqrt{2\pi}\omega(-\tau/\theta)|\tau|^{-1/2} e^{i\tau \log \left|\frac{\tau}{e\theta}\right|} e^{i\frac{\pi}{4} \sgn(\theta)} + O(|\tau|^{-3/2}).
$$Further, if $|\tau|\geq |\theta N|^{1+\epsilon}$, then
$$I \ll_{A, \epsilon} |\tau|^{-A},
$$and if $|\tau| \leq |\theta N|^{1-\epsilon}$,then
$$I \ll_{A, \epsilon} |\theta N|^{-A}
$$
\end{lemma}

\begin{rmk1}\label{rmk:integral}
Also, we note that if $|\tau |\leq 1$, then $I \ll_A |\theta N|^{-A}$ and if $|\theta N|\leq 1$, then $I\ll_A (1+|\tau|)^{-A}$.  
\end{rmk1}
We refer the reader to \cite{LY} for the proofs of the preceding statements.  

In further analysis of the exponential integral, we will see that sometimes the sum is localized to very short intervals.  We record the following easy Lemma for convenience.

\begin{lemma}\label{lem:shortsum}
Let $A \geq B >0$.  Then,
$$\sum_{A\leq n\leq A+B} \frac{|A(1, n)|}{n} \ll \bfrac{B}{A}^p A^\epsilon \con^\epsilon,
$$where we have $p = 1$ if the Ramanujan conjecture holds, and $p=1/2$ unconditionally.
\end{lemma}
\begin{proof}
If Ramanujan holds, then 
$$\sum_{A\leq n\leq A+B} \frac{|A(1, n)|}{n} \ll A^\epsilon \log\bfrac{A+B}{A},
$$from which the conclusion follows.  Otherwise by Cauchy's inequality,
\begin{eqnarray*}
\sum_{A\leq n\leq A+B} \frac{|A(1, n)|}{n} 
\leq \left(\sum_{A\leq n\leq A+B} \frac{|A(1, n)|^2}{n} \right)^{1/2}\left(\sum_{A\leq n\leq A+B} \frac{1}{n}\right)^{1/2}\\
\ll A^\epsilon \con^\epsilon \log\bfrac{A+B}{A}^{1/2},
\end{eqnarray*}from which the claim follows.  Here we have used that 
$$\sum_{A\leq n\leq A+B} \frac{|A(1, n)|^2}{n} \ll A^\epsilon \con^\epsilon,$$ which follows by the convexity bound for Rankin-Selberg L-functions $L(s, F\times\tilde{F})$.  Brumley \cite{Br} proved this convexity bound for $L(s,F)$ automorphic for $GL(n)$ for $n\leq 4$ using recent progress in functoriality and the author \cite{Li} proved this for all $n$ by a different method.
\end{proof}

\subsection{Preliminary cleaning}
Let 
$$G(s) = \frac{\Gamma\bfrac{s+k}{2}}{\Gamma\bfrac{1-s+k}{2}}.
$$The $\Gamma$ factors which appear in the integral transform $\Psi_k$ is $\G(1+s)$ where
$$\G(s) = G(s+\alpha_1)G(s+\alpha_2)G(s+\alpha_3)
$$and the Langlands parameters $\alpha_i$ satisfy $\alpha_1 + \alpha_2 + \alpha_3 = 0$, and $\tRe \alpha_1 = \tRe \alpha_2 = \tRe \alpha_3 = 0$ by temperedness.  Thus set $\alpha_j = i a_j$ for $a_j \in \mathbb{R}$.  

Then, for $\sigma > -1$, Stirling's approximation gives 
$$G(1+\sigma + it) \ll (1+|t|)^{\sigma + 1/2}
$$so that
$$\G(1+\sigma + it) \ll \big((1+|t+a_1|)(1+|t+a_2|)(1+|t+a_3|)\big)^{\sigma + 1/2}.
$$Recalling that
$$\con = (1+|\alpha|)(1+|\beta|)(1+|\gamma|),
$$we have
$$\G(1+\sigma +it) \ll \big( \con(|t|+1)+ |t|^3\big)^{\sigma + 1/2}.
$$
By Lemma \ref{lem:integral} and Remark \ref{rmk:integral}, $\tilde{\psi}(x) \ll_A  (xN)^{-\sigma} \left(1+\frac{|t|}{1+|\theta N|^{1+\epsilon}}\right)^{-A}$.  Thus, for $\sigma > -1$,
\begin{eqnarray}\label{eqn:Psibound}
\Psi_k(x) 
&\ll_{\sigma, A}& \int_{-\infty}^\infty (xN)^{-\sigma} \left(1+\frac{|t|}{1+|\theta N|^{1+\epsilon}}\right)^{-A} \big( \con(|t|+1)+ |t|^3\big)^{\sigma + 1/2} \notag \\
&\ll& (xN)^{-\sigma} \big( \con(|\theta N|+1)+ |\theta N|^3\big)^{\sigma + 1/2} (1+|\theta N|^{\epsilon}).
\end{eqnarray}We first record the following results on $\Psi_k(x)$.

\begin{lemma}\label{lem:prelimPsi}
Let $U = \con(|\theta N|+1)+ |\theta N|^3$.  
\begin{enumerate}
\item If $xN \geq N^{\epsilon} U$, then $\Psi_k(x) \ll_A N^{-A}$ for any $A>0$.
\item If $|\theta N| \ll \con^\epsilon$, then $\Psi_k(x) \ll \con^{1/2+\epsilon}$.  
\item Let $R_2 = \{ t\in \mathbb{R}: |t+a_i|\geq \con^\epsilon \textup{ for all } 1\leq i\leq 3\}$.  Further, let 
\begin{align*}
f(t) &= t\log \bfrac{\pi^3 x |t|}{e|\theta|} - (t+a_1) \log \bfrac{|t+a_1|}{2e} \\
&-(t+a_2) \log \bfrac{|t+a_2|}{2e} -(t+a_3) \log \bfrac{|t+a_3|}{2e}.
\end{align*}
If $|\theta N| \geq \con^\epsilon$, then there exists a smooth function $g(t)$ with support when $|t| \asymp |\theta N|$ satisfying $\frac{d^j}{dt^j} g(t) \ll_j |t|^{-1/2-j}$ such that
$$\Psi_k(x)  \ll \sqrt{xN} \int_{R_2}g(t) e^{if(t)}dt + \sqrt{xN}|\theta N|^{-1/2+\epsilon} \con^\epsilon.
$$

\end{enumerate}
\end{lemma}
\begin{proof}
If $xN \geq N^{\epsilon} U$, then shift contours to the right to see that the integral is $\ll_A N^{-A}$ for any $A>0$.  Now, if $|\theta N|\ll \con^{\epsilon}$, the desired bound follows from (\ref{eqn:Psibound}) upon setting $\sigma = 0$.  

Hence assume that $|\theta N| \geq \con^{\epsilon}$.  We restrict our attention to the range $|\theta N|^{1-\epsilon} \leq |t| \leq |\theta N|^{1+\epsilon}$, since otherwise, $\tilde{\psi}(s)$ is very small by Lemma \ref{lem:integral}.  Set $\sigma = -1/2$.  Then by Lemma \ref{lem:integral}, we have that
$$\tilde{\psi}(x) = \sqrt{2\pi xN} \omega\bfrac{-t}{\theta}|t|^{-1/2} e^{it \log \left|\frac{t}{e\theta}\right|} e^{i \frac{\pi}{4} \sgn(\theta)}+O(|t|^{-3/2}).
$$The contribution of $O(|t|^{-3/2})$ to the integral $\Phi_k(x)$ is
$$\ll \sqrt{xN} \int_{|\theta N|^{1-\epsilon}\leq |t| \leq |\theta N|^{1+\epsilon}} |t|^{-3/2} dt \ll \sqrt{xN}|\theta N|^{-1/2+\epsilon}.
$$We now seek to understand the contribution from the main term, which up to a constant factor is
\begin{equation}\label{eqn:applysaddlepsi}
\sqrt{xN} \int_{-\infty}^\infty (x\pi^3)^{it} \omega \bfrac{-t}{\theta} |t|^{-1/2} e^{it \log \left|\frac{t}{e\theta}\right|} \G(1+\sigma-it) dt
\end{equation}
Stirling's approximation gives that us that
$$G(1/2 - it) = e^{-it \log |t/2e|} \left(c_0 + \frac{c_1}{|t|}+...+O\bfrac{1}{|t|^A}\right),
$$where the $c_i$ are absolute constants.  We split the integral in (\ref{eqn:applysaddlepsi}) into two ranges $R_1$ and $R_2$, where $R_1 = \{t\in \mathbb{R}: |t+a_i| \geq \con^{\epsilon} \textup{ for some } 1\leq i\leq 3\} $ and $R_2$ is the complement of $R_1$.  The contribution of $R_1$ gives $\ll \sqrt{xN}|\theta N|^{-1/2} \con^{\epsilon}$.  For $R_2$, we use Stirling's approximation for $\G$ to get that (\ref{eqn:applysaddlepsi}) can be rewritten as
$$\sqrt{xN} \int_{R_2} g(t) e^{if(t)} dt,
$$as desired.
\end{proof}

\section{Proof of Theorem \ref{thm:smooth}}
We will be deriving various bounds for $\Psi_k(x)$ in this section, and it will be convenient to record the contributions these make to $\T_k$ below.  Note that $x = \frac{n}{q^3}$ and $U^\epsilon \ll (QN)^\epsilon$.  Since Theorem \ref{thm:smooth} is trivial otherwise, we also assume that $\con^\epsilon \ll N^\epsilon$.  
\begin{align} \label{eqn:sumbound}
&q^{3/2+\epsilon} \sum_{xN\leq UN^\epsilon} \frac{|A(n, 1)|}{n} \left(\con^{1/2} + \sqrt{U} |\theta N|^{-1/2} + |\theta N|^{3/2}\right) \\
&\ll (QN)^{\epsilon} \left(Q^{3/2} \con^{1/2} + q^{3/2} \left(\con^{1/2} \sqrt{|\theta N| + 1} + |\theta N|^{3/2} \right) |\theta N|^{-1/2} + q^{3/2} |\theta N|^{3/2} \right) \notag \\
&\ll (QN)^{\epsilon} \left( Q^{3/2} \con^{1/2} + \bfrac{N}{Q}^{3/2}\right). \notag 
\end{align}

In particular, we see that the contribution of the terms from parts (1) and (2) of Lemma \ref{lem:prelimPsi} and from the error term from part (3) of Lemma \ref{lem:prelimPsi} to $\T_k$ is bounded by the above.  Let 
$$J = \sqrt{xN} \int_{R_2} g(t) e^{if(t)} dt,
$$with notation as in Lemma \ref{lem:prelimPsi}.  In order to prove cancellation in this integral, our first step is to record some expressions for $f'(t)$ and $f''(t)$.  

Without loss of generality, assume that $|a_1|\geq |a_2| \geq |a_3|$. Note that $a_1 \asymp a_2$, so $\con^{1/3} \leq a_1 \leq \con^{1/2}$.  For future use, let
\begin{equation}\label{eqn:C}
C(t) := \prod_i (t+a_i) = t^3 - (a_1^2 - a_2a_3)t + a_1a_2a_3,
\end{equation}since $\sum_i a_i = 0$.  For $|\theta N|\asymp |t|$, we have that
\begin{equation}\label{eqn:boundC}
C(t) \ll (|\theta N|+1) (|\theta N|^2 + \con).
\end{equation}

Further, after some calculations,
\begin{equation}\label{eqn:fprime}
 f'(t) = \log \bfrac {8\pi^3 xN|t|}{|\theta N| \prod_i |t+a_i|} = \log \bfrac {8\pi^3 xN|t|}{|\theta N C(t)|}.
\end{equation}
Using the fact that $\sum_i a_i = 0$,
\begin{align}\label{eqn:f''}
f''(t) 
&= \frac{1}{t} - \sum_i \frac{1}{t+a_i}  \\
&= \frac{\prod_i (t+a_i) - t\left((t+a_1)(t+a_2)+(t+a_3)(t+a_1)+(t+a_2)(t+a_3)\right)}{t\prod_i(t+a_i)} \notag \\
&= \frac{\prod_i (t+a_i) - t\left(3t^2 + \sum_{i<j} a_i a_j \right)}{t\prod_i(t+a_i)} \notag \\
&= \frac{t^3 +t\left(\sum_{i<j} a_i a_j\right) + \prod_i a_i - 3t^3 - t\left(\sum_{i<j} a_i a_j \right)}{t\prod_i(t+a_i)} \notag \\
&= \frac{\prod_i a_i - 2t^3}{t\prod_i(t+a_i)} \notag \\
&= \frac{\prod_i a_i - 2t^3}{tC(t)}.\notag
\end{align} 

We now consider different ranges of $|\theta N|$.  We consider the case $|\theta N|\gg \con^{1/2-\epsilon}$ in \S 3.1, $\con^{1/3}\leq |\theta N| \leq \con^{1/2-\epsilon}$ in \S 3.2, and $\con^\epsilon < |\theta N| < \con^{1/3}$ in \S 3.3.

\subsection{$|\theta N| \gg \con^{1/2-\epsilon}$}
Here, since we may assume that $|t| \asymp |\theta N|$, we have $|t| \gg \con^{1/2 - \epsilon}$.  In this case, since $t^3 \gg \con^{3/2-\epsilon}$ and $\left|\prod_i a_i\right| \leq \con$,
$$f''(t) \asymp \frac{t^2}{C(t)} \gg \frac{t^2}{t^3}\con^{-\epsilon},
$$by (\ref{eqn:C}).  Thus $|f''(t)| \gg \frac{1}{|\theta N| \con^\epsilon}$ and by Lemma 5.1.3 of \cite{Hu},
$$\int_{\alpha}^{\beta} g(t) e^{if(t)}dt \ll \frac{1}{|\theta N|^{1/2}} \sqrt{|\theta N| \con^{\epsilon}} \ll |\theta N|^{\epsilon}.
$$Thus the contribution to $\T_k$ is bounded by 
\begin{eqnarray*}
&\ll& q^{3/2+\epsilon} \sum_{xN\leq UN^\epsilon} \frac{|a(n)|}{n} \sqrt{xN} |\theta N|^{\epsilon}\\
&\ll& q^{3/2+\epsilon} \sqrt{U}N^{\epsilon} \\
&\ll& q^{3/2+\epsilon} N^{\epsilon} |\theta N|^{3/2} \\
\end{eqnarray*}where we have used $|\theta N|\gg \con^{1/2-\epsilon}$ to see that $U \ll |\theta N|^3 \con^\epsilon\ll |\theta N|^3 N^\epsilon$.  The latter is bounded by (\ref{eqn:sumbound}).

\subsection{$\con^{1/3} \leq |\theta N| \leq \con^{1/2 -\epsilon}$}
Let 
$$\Delta = xN - \frac{|\theta N|\prod_i |t+a_i|}{8\pi^3 |t|}.
$$We may write
$$f'(t) = \log \left(1+\frac{8\pi^3|t|\Delta}{|\theta N| \prod_i |t+a_i|}\right).
$$Now, if $xN \not \asymp \frac{|\theta N| \prod_i |t+a_i|}{|t|}$, then we are done, then then $f'(t) \gg 1$, and $J\ll \frac{\sqrt{xN}}{|\theta N|^{1/2-\epsilon}}$ by Lemma 5.1.2 of \cite{Hu}.  Then the contribution to $\T_k$ is $\ll q^{3/2+\epsilon}\frac{\sqrt{U}}{|\theta N|^{1/2-\epsilon}}$, which is bounded as in (\ref{eqn:sumbound}).

Thus, assume that
$$xN \asymp \frac{|\theta N|\prod_i |t+a_i|}{|t|} \asymp |C(t)| \ll |\theta N|\con,
$$for some $t \asymp |\theta N|$, where we have used (\ref{eqn:boundC}).  Note that $f'(t) \asymp \frac{\Delta}{C(t)}$.  We proceed differently according to the size of $f'(t)$.

\subsubsection{$f'(t)$ is small}
Suppose that
$$\Delta \ll |\theta N|^3,
$$for some $t\asymp |\theta N|$.
In this case $\prod_i a_i - 2t^3 \asymp t^3$, since $t^3 \geq \con \geq \prod_i |a_i|$.  Then
$$f''(t) \asymp \frac{t^2}{C(t)}.
$$Now for $M\geq 1$, let $I_M = C^{-1}([M, 2M)\cup(-2M, -M])$ and 
$$J_M = \sqrt{xN} \int_{I_M} g(t) e^{if(t)}dt.$$   
Trivially, $I_M$ is always a union of 6 intervals or less.  From Lemma \ref{lem:prelimPsi}, by the definition of $R_2$, $C(t) \gg \con^\epsilon$ for $t\in R_2$ so we may assume that $M \gg \con^\epsilon$.  Fix $M$, and assume that $xN \asymp \frac{|\theta N|C(t)}{|t|} \asymp C(t) \asymp M$ for $t\in I_M$, since otherwise, we have that $f'(t) \gg 1$ and $J_M \ll \frac{1}{\sqrt{|\theta N|}}$ by Lemma 5.1.2 of \cite{Hu} as before.  In particular, we need only consider one value of $M$ in the sequel.  In this case $f''(t) \asymp \frac{|\theta N|^2}{M}$ so $J_M \ll \frac{\sqrt{xN}}{\sqrt{|\theta N|}}\frac{\sqrt{M}}{|\theta N|} = \frac{\sqrt{xNM}}{|\theta N|^{3/2}}$ by Lemma 5.1.3 in \cite{Hu}.

The contribution of this to $\T_k$ is bounded by
\begin{eqnarray*}
S_M &:=& q^{3/2+\epsilon}\sum_{A\leq n\leq A+B} \frac{|a(n)|}{n} \frac{\sqrt{xNM}}{|\theta N|^{3/2}}\\
&\ll& \frac{MN^\epsilon}{|\theta N|^{3/2}} q^{3/2+\epsilon} \bfrac{B}{A}^p,
\end{eqnarray*}by Lemma \ref{lem:shortsum}, where $p = 1/2$ unconditionally, and $p=1$ on Ramanujan.  Since $xN \asymp M$, $A \asymp \frac{q^3 M}{N}$, and $B \ll \frac{|\theta N|^3 q^3}{N}$.  For $p=1$, $S_M \ll \frac{MN^\epsilon}{|\theta N|^{3/2}} q^{3/2+\epsilon} \frac{|\theta N|^3}{M} = q^{3/2+\epsilon}N^\epsilon |\theta N|^{3/2+\epsilon}$, which is bounded by the right hand side of (\ref{eqn:sumbound}).  Unconditionally, when $p=1/2$, 
\begin{equation}\label{eqn:SMunconditionalbound1}
S_M \ll M^{1/2}q^{3/2+\epsilon} \ll N^\epsilon\sqrt{\con \theta N} q^{3/2+\epsilon} \leq Q^{\epsilon}N^\epsilon\sqrt{\con N Q}.
\end{equation}

\subsubsection{$f'(t)$ is large}  Here, suppose that $Y\leq \Delta < 2Y$ for some $Y \geq |\theta N|^3$.  Then 
$$|f'(t)| \asymp \frac{|t| \Delta}{|\theta N||C(t)|} \asymp \frac{Y}{|C(t)|}.
$$Again split the integral into $J_M$ as before.  $J_0$ is bounded exactly as above.  Note that $J_M \ll \frac{M}{Y\sqrt{|\theta N|}}$ by Lemma 5.1.2 of \cite{Hu}. Then for $A \asymp \frac{q^3}{N} M$ and $B \asymp \frac{q^3}{N} Y$, we have
\begin{eqnarray*}
S_M &\ll& q^{3/2+\epsilon} \sum_{A\leq n\leq A+B} \frac{|a(n)|}{n} \sqrt{M}  \frac{M}{Y\sqrt{|\theta N|}} \\
&\ll& N^\epsilon q^{3/2+\epsilon} \frac{M^{3/2}}{Y\sqrt{|\theta N|}} \bfrac{Y}{M}^p,
\end{eqnarray*}
by Lemma \ref{lem:shortsum}.  Assuming Ramanujan, we have $p=1$.  Since $M \leq \con |\theta N|$, $S_M\ll q^{3/2+\epsilon} \frac{M^{1/2}}{\sqrt{|\theta N|}} \ll q^{3/2+\epsilon} \sqrt{\con}$ which is bounded by (\ref{eqn:sumbound}).

Unconditionally we have $p=1/2$.  Using that $Y\geq |\theta N|^3$, $|\theta N| \geq \con^{1/3}$, we have
\begin{align}\label{eqn:SMunconditionalbound2}
S_M &\ll q^{3/2+\epsilon}N^\epsilon \frac{M^{3/2}}{Y\sqrt{|\theta N|}}\bfrac{Y}{M}^{1/2} \\
&\ll N^\epsilon\frac{q^{3/2}\con}{|\theta N|} \leq N^\epsilon q^{3/2}\con^{2/3}.\notag
\end{align}

\subsection{$\con^\epsilon < |\theta N| < \con^{1/3}$}  

If $xN \not \asymp \frac{|\theta N|C(t)}{t} \asymp C(t)$ for $t\asymp |\theta N|$, then we are done as before since then $f'(t) \gg 1$.  Hence assume that $xN \asymp C(t)$ for some $t \asymp |\theta N|$.  Define $J_M$ as in the last section.  If $M \not \asymp xN$, we are similarly done, so assume that $C(t) \asymp M \asymp xN$.  We split into two cases.

\subsubsection{$\Delta \ll \frac{C(t)}{|\theta N|}$}  The trivial bound gives $J_M \ll \sqrt{xN} |\theta N|^{1/2+\epsilon}$, which contributes  
\begin{eqnarray*}
&\ll& q^{3/2+\epsilon} \sum_{A\leq n\leq A+B} \frac{|a(n)|}{n} \sqrt{M}|\theta N|^{1/2+\epsilon}\\
&\ll& N^\epsilon q^{3/2+\epsilon}\sqrt{M}|\theta N|^{1/2+\epsilon} \bfrac{B}{A}^p,
\end{eqnarray*}by Lemma \ref{lem:shortsum} where $A \asymp C(t) \frac{q^3}{N}$ and $B \ll \frac{C(t)}{|\theta N|}\frac{q^3}{N}$.  Say that $p=1$.  Using $M \ll |\theta N| \con$, this leads to $S \ll N^\epsilon \con^{1/2} q^{3/2+\epsilon} |\theta N|^{\epsilon}$ which is bounded by (\ref{eqn:sumbound}).

In the unconditional case, $p=1/2$, so we have
\begin{equation} \label{eqn:SMunconditionalbound3}
q^{3/2+\epsilon}\sqrt{M} \ll N^\epsilon Q^{\epsilon}\sqrt{Nq\con}.
\end{equation}

\subsubsection{$\Delta \gg \frac{C(t)}{|\theta N|}$}  Here we again split the range for $\Delta$ into diadic intervals.  Let $Y<\Delta \leq 2Y$.  We have that $f'(t) \asymp \frac{\Delta}{C(t)} \asymp \frac{Y}{M}$.  Thus
$$J_M \ll \frac{1}{\sqrt{\theta N}} \frac{M}{Y}.
$$Then
$$S_M \ll q^{3/2+\epsilon} \sum_{A\leq n\leq A+B} \frac{|a(n)|}{n} \frac{\sqrt{M}}{\sqrt{\theta N}}\frac{M}{Y}
\ll N^\epsilon q^{3/2+\epsilon}\frac{\sqrt{M}}{\sqrt{\theta N}} \sqrt{M}\frac{M}{Y}\bfrac{B}{A}^p
$$where $A \asymp \frac{q^3}{N} M$ and $B\asymp \frac{q^3}{N} Y$.  Thus for $p=1$,
$$S_M \ll N^\epsilon q^{3/2+\epsilon}\frac{1}{\sqrt{\theta N}} \sqrt{M} \ll \frac{N^\epsilon}{\sqrt{\theta N}}q^{3/2+\epsilon} \sqrt{\con |\theta N|} \ll N^\epsilon \sqrt{\con} Q^{3/2+\epsilon},
$$which is bounded by (\ref{eqn:sumbound}).

For $p=1/2$, we get
$$S_M \ll N^\epsilon \frac{q^{3/2+\epsilon} \sqrt{M}}{\sqrt{|\theta N|}} \bfrac{M}{Y}^{1/2}.$$  Since $Y \gg \frac{M}{|\theta N|}$, this leads to 
\begin{equation}\label{eqn:SMunconditionalbound4}
S_M \ll q^{3/2+\epsilon} \sqrt{M}\ll Q^{\epsilon} \sqrt{Nq\con}
\end{equation}

\subsection{Conclusion}
From (\ref{eqn:sumbound}) and the sections above, we have that on Ramanujan,
$$\T_k \ll (QN)^{\epsilon} \left(Q^{3/2} \con^{1/2} + NQ^{-1/2}+ \bfrac{N}{Q}^{3/2}\right) \ll N^{3/4 + \epsilon} \con^{1/4},
$$upon setting $Q = \frac{N^{1/2}}{\con^{1/6}}$.

By (\ref{eqn:sumbound}),(\ref{eqn:SMunconditionalbound1}), (\ref{eqn:SMunconditionalbound2}),(\ref{eqn:SMunconditionalbound3}) and (\ref{eqn:SMunconditionalbound4}), we have that the unconditional bound has two extra terms so that for $Q = \frac{N^{1/2}}{\con^{1/6}}$,
$$\T_k \ll N^\epsilon Q^{\epsilon}\sqrt{Nq\con} + N^\epsilon Q^{3/2}\con^{2/3} + N^{3/4 + \epsilon} \con^{1/4} \ll N^{3/4 + \epsilon} \con^{5/12}.
$$

\paragraph{\bf Acknowledgements:} I would like to thank Xiaoqing Li for commenting on a preprint of this paper.

\end{document}